\documentclass[12pt]{amsart}

\newcommand{\C}{\mathbb{C}} %% complex numbers

\numberwithin{equation}{section}

\newtheorem{theorem}{Theorem}[section]

\newtheorem{lemma}[theorem]{Lemma}

\newtheorem{question}[theorem]{Question}

\author[Gauthier, Knese]{P. M. Gauthier and Greg Knese}

\address{D\'epartement de math\'ematiques et de statistique, Universit\'e de Montr\'eal,
CP-6128 Centreville, Montr\'eal,  H3C3J7, CANADA}
\email{gauthier@dms.umontreal.ca}

\address{Department of Mathematics, University of Alabama, Box 870350
Tuscaloosa, AL 35487-0350}
\email{geknese@bama.ua.edu}


\date{\today}

\begin{document}
\begin{abstract}
Sur un compact du plan dont le compl\'ementaire est connexe, est-il
possible d'approcher uniform\'ement une fonction continue, holomorphe
et sans z\'eros \`a l'int\'erieur, par des polyn\^omes n'ayant aucun
z\'eros sur le compact tout entier? Dans cette note br\`eve, nous
rappelons le raport surprenant entre ce probl\`eme et l'hypoth\`ese de
Riemann et donnons une r\'eponse affirmative pour une ``chaine" de
domaines de Jordan.

On a compact subset of the plane with connected complement, is it
possible to uniformly approximate a continuous function, holomorphic and
non-vanishing on the interior, with polynomials non-vanishing on the
entire compact set? In this brief note, we recall the surprising
connection between this question and the Riemann hypothesis and
proceed to provide an affirmative answer for a ``chain'' of Jordan
domains.
\end{abstract}

\title[Zero-free approximation]{Zero-free polynomial approximation on a chain of Jordan domains}

\maketitle
\section{introduction}

For a compact set $K\subset\C,$ we denote by $A(K)$ the family of
continuous functions on $K,$ which are holomorphic on the interior
$K^o$ of $K.$ Mergelyan's theorem asserts that every $f\in A(K)$ is
uniformly approximable by polynomials if and only if $\C\setminus K$
is connected.

\begin{question}\label{polynomial-question}
Let $K$ be compact subset of $\C$ with connected complement. Suppose
$f\in A(K)$ has no zeros on $K^o$ and $\epsilon>0.$ Is there a
polynomial $p_\epsilon$ with no zeros on $K$ such that $\max_{z\in
  K}|f(z)-p_\epsilon(z)|<\epsilon$?
\end{question}

An affirmative answer has been given when $K$ is strictly starlike
\cite{pG10}, a closed Jordan domain \cite{jA11}, or a disjoint union
of such compacta \cite{pG12}.  In this note, we investigate the case
when $K$ is a union of finitely many Jordan domains, not necessarily
disjoint.  Question \ref{polynomial-question} is related to the
following question regarding approximation by vertical translates of
the Riemann zeta-function.

\begin{question}\label{zeta-question}
Let $K$ be a compact subset of the strip $1/2<\Re(z)<1$ with connected
complement. Suppose $f\in A(K)$ has no zeros on $K^o$ and
$\epsilon>0.$ Is the set of $t>0$, such that $\max_{x\in K}
|f(z)-\zeta(z+it)|<\epsilon$, of positive lower density?
\end{question}

Recently, Johan Andersson has made the remarkable observation
\cite{jA11} that these two problems are equivalent.  Under the
stronger hypothesis that $f$ has no zeros on $K$ (rather than on
$K^o$), the answer to Question \ref{polynomial-question} is positive,
as an obvious consequence of Mergelyan's Theorem.  Under this stronger
hypothesis, Question \ref{zeta-question} also has a positive answer,
however this is far from obvious. It is a consequence of Voronin's
spectacular universality theorem for the Riemann zeta-function, which
has been refined by Bhaskar Bagchi \cite{bB82} and Steven Mark Gonek
\cite{sG79}.

For a measurable set $E$ of positive numbers, we denote by $m(E)$ the
measure of $E$ and by $\underline d(E)$ and $\overline d(E)$
respectively the lower and upper densities of $E$
$$ \underline d(E)=\liminf_{T\rightarrow\infty}\frac{m(E\cap[0,T])}{T}
\qquad \overline
d(E)=\limsup_{T\rightarrow\infty}\frac{m(E\cap[0,T])}{T}.
$$

The following result of Bagchi suggests that these problems may be
related to the Riemann Hypothesis and therefore might be difficult to
solve in complete generality.

\begin{theorem}[Bagchi] The following assertions are equivalent.

1) The Riemann hypothesis is true.

2)For each compact set $K$ with connected complement lying in the
strip $1/2 < Re(z) < 1$ and for each $\epsilon > 0$,
$$ \overline d\left(\{t>0:\max_{z\in
  K}|\zeta(z+it)-\zeta(z)|<\epsilon\}\right)>0.
$$

3) For each compact set $K$ with connected complement lying in the
strip $1/2 < Re(z) < 1$ and for each $\epsilon > 0$,
$$ \underline d\left(\{t>0:\max_{z\in
  K}|\zeta(z+it)-\zeta(z)|<\epsilon\}\right)>0.
$$
\end{theorem}

For a further discussion of this issue, we refer to \cite{pG12}.

\section{Chain of Jordan domains}

Our main theorem is the following.

\begin{theorem}
Let $\Omega=\Omega_1\cup\Omega_2\cup\cdots\cup\Omega_n$ be a chain of
Jordan domains. That is,
$\overline\Omega_i\cap\overline\Omega_j=\emptyset$ if $|i-j|>1$ and
$\overline\Omega_i\cap\overline\Omega_j$ is a single point if
$|i-j|=1.$ Suppose $f\in A(\overline\Omega)$ and $f(z)\not=0,$ for
$z\in\Omega.$ Then, for each $\epsilon>0,$ there is a polynomial
$p_\epsilon$ such that $|f-p_\epsilon|<\epsilon$ and
$p_\epsilon(z)\not=0,$ for $z\in\overline\Omega.$
\end{theorem}

Let $D_1 = \{z: |z+1/2|<1/2\}, D_2 = \{z:|z-1/2|<1/2\}$.

We introduce three methods of approximation via three lemmas (whose
proofs are trivial). We frequently use $f^{-1}(0)$ in place of
$f^{-1}(\{0\})$.

\begin{lemma}\label{shrinking}
Let $D$ be a disk, $p \in \partial D, f\in A(\overline D)$, and
$f^{-1}(0)\subset\partial D.$ Then, for each $\epsilon>0,$ there
exists $f_\epsilon\in A(\overline D)$ with $|f-f_\epsilon|<\epsilon,
f^{-1}(0)\subset\{p\}, f_\epsilon(p)=f(p).$ We say that $f_\epsilon$
is an approximation via shrinking toward $p$.
\end{lemma}

\begin{proof}
We may assume $D=D_2$ and $p=0$.  Set $f_\epsilon(z)=f(rz),$ for
$0<r<1,$ and $r$ sufficiently near $1.$
\end{proof}

\begin{lemma}\label{lens}
Let $f\in A(\overline D_2), f^{-1}(0)\subset\partial D_2.$ Then, for
each $\epsilon>0,$ there exists $f_\epsilon\in A(\overline D_2)$ such
that $|f-f_\epsilon|<\epsilon, f^{-1}(0)\subset\{0,1\},
f_\epsilon(0)=f(0), f_\epsilon(1)=f(1).$
\end{lemma}

\begin{proof}
For $0<r<1$ the mapping $L_r:D_2 \to D_2$
\[
L_r(z) =
\frac{\left(\frac{z}{1-z}\right)^r}{1+\left(\frac{z}{1-z}\right)^r}
\]
maps the disk $D_2$ onto a lens shaped region with corners at the points
$0$ and $1$ of angle $\pi r$. Here the $r$-th root is chosen so that
$1^r=1$.

Set $f_\epsilon = f\circ L_r$ with $r$ sufficiently close to $1$.
\end{proof}

\begin{lemma}\label{parabolic}
Let $f\in A(\overline D_2)$ and $\epsilon>0.$ For
$\delta=\delta(\epsilon)>0,$ set
$$ f_\epsilon(z) = f(w), \quad\mbox{where}\quad w =
\left(1-\frac{z-1}{z}-i\delta\right)^{-1}.
$$ Then, $f_\epsilon(0)=f(0)$ and for sufficiently small $\delta,$ we
have $|f_\epsilon-f|<\epsilon.$ We call such an $f_\epsilon$ an
approximation of $f$ on $\overline D_2$ parabolic at $0.$
\end{lemma}

\begin{proof}
Notice that for $\delta=0,$ we have $f_\epsilon=f$. It is then clear
that $f_\epsilon \to f$ uniformly as $\delta \to 0$.
\end{proof}

Removing a zero at a point of contact between two disks is the most
technical part of our proof.

\begin{lemma}\label{2discs}
Let $D = D_1\cup D_2$. Let $f\in A(\overline D),
f^{-1}(0)\subset\{0,1\}.$ Then, there exists $f_\epsilon\in
A(\overline D)$ such that $|f-f_\epsilon|<\epsilon,
f_\epsilon^{-1}(0)\subset\{1\}, f_\epsilon(\pm 1)=f(\pm 1).$
 \end{lemma}

\begin{proof}
If $f(0)\not=0,$ there is nothing to prove. Set $f_\epsilon=f.$

Suppose $f(0)=0.$ Then, $f$ is constant on neither $D_1$ nor $D_2,$
for otherwise $f$ would have interior zeros, contrary to the
hypothesis. Set $f_j=f\mid\overline D_j, j=1,2.$ As in the proof of
Lemma \ref{shrinking}, let $g_1$ be an approximation of $f_1$ by
shrinking $\overline D_1$ towards $-1$ and let $g_2$ be an
approximation of $f_2$ by shrinking $\overline D_2$ towards $+1.$ We
choose the approximation so that $|f_j-g_j|<\epsilon/3$ on $\overline
D_j.$ We note that $g_j\in A(\overline D_j), \, g_1^{-1}(0)=\emptyset,
\, g_2^{-1}(0)\subset\{+1\}, \, g_1(-1)=f_1(-1), \, g_2(+1)=f_2(+1).$
If $g_1(0)=g_2(0),$ we may set $f_\epsilon=g_j$ on $\overline D_j,$
for $j=1,2$ and the proof is complete.

Suppose $g_1(0)\not=g_2(0).$ Let $h_\epsilon$ be an approximation of
$g_1$ on $\overline D_1$ parabolic at $-1$ in the sense of Lemma
\ref{parabolic}. Note that $h_\epsilon(\overline D_1)=g_1(\overline D_1),$ so $h_\epsilon$ omits zero on $\overline D_1.$
We claim that there are arbitrarily close such
approximations such that $h_\epsilon(0)$ is not on the line determined
by $0$ and $g_2(0)$.  If not, it follows from the construction of
$h_\epsilon$ that $g_1(z)$ is on this line, for all $z\in\partial D_1$
near $0.$ Therefore, $g_1(\partial D_1)$ is in this line. This can be
proved via conformal mapping using the fact that a function analytic
on a neighborhood of the closed upper half plane and real valued on an
interval of the real line must be real valued on the entire real line.
Consequently, $g_1(\overline D_1)$ is also in this line. But $g_1$ is
non-constant and hence open on $D_1,$ which is a contradiction. Thus,
we may choose $h_\epsilon$ such that $|h_\epsilon-g_1|<\epsilon/3$ on
$\overline D_1$ and $h_\epsilon(0)$ is not on the line determined by
$0$ and $g_2(0).$

Note that $|g_2(0)-h_{\epsilon}(0)| \leq |g_2(0)-f(0)|+ |f(0)-g_1(0)| +
|g_1(0)-h_{\epsilon}(0)| < \epsilon$.  Let us write
$g_2(0)-h_\epsilon(0)$ in polar coordinates:
\[
 g_2(0)-h_\epsilon(0) = re^{i\alpha},
\]
 where $r<\epsilon.$ We note that $0$ is not on the line segment
\[
 h_\epsilon(0)+te^{i\alpha}, \quad 0\le t\le r
\]
by choice of $h_\epsilon$.
Consider the pie piece:
\[ P = \{te^{i(\alpha+\varphi)}:  0\le t\le r,
|\varphi|\le\delta_1\}.
\]
Choose $\delta_1>0$ so small that $0$ is not on the translated
closed pie piece given by
\[ h_\epsilon(0)+ P.
\]
 By the continuity of $h_\epsilon,$ there is a $\delta_2>0$ such
that $0$ is not in the set
\begin{equation}\label{zero-free-pie}
    h_\epsilon(z)+P, \quad |z|\le\delta_2, \quad z\in\overline D_1.
\end{equation}

Let us define a mapping $w=\eta(z)$ on $\overline D_1$ via a series of
transformations
\[
\begin{aligned}
z&\mapsto z_1 = -\frac{z+1}{z},  &(D_1 \to RHP:=\text{right
  half plane})&\\
 z_1&\mapsto z_2 = z_1^{2\delta_1/\pi}, \quad z_2(1)=1,
&(RHP \to
\text{ sector with angle } 2\delta_1)&\\
z_2&\mapsto z_3 = \delta_3z_2, \quad \delta_3>0, &\text{(contraction
  of the sector)}&\\
z_3&\mapsto z_4 = r\frac{z_3}{z_3+1}, &(\text{sector} \to
\text{lens})&\\
z_4&\mapsto w = e^{i\alpha}z_4. &\text{(rotation of the lens)}&
\end{aligned}
\]
Thus, $\eta \in A(\bar{D}_1)$ maps $\overline D_1$ to a ``lens" of
angular opening $2\delta_1,$ whose end points are $\eta(-1)=0$ and
$\eta(0) = re^{i\alpha}.$ The parameter $\delta_3$ will be chosen
momentarily.

Define $f_\epsilon(z)=h_\epsilon(z)+\eta(z),$ for $z\in\overline D_1$
and $f_\epsilon(z)=g_2(z),$ for $z\in\overline D_2.$ Then,
$f_\epsilon(z)\in A(\overline D)$ since $\eta(0) =
g_2(0)-h_{\epsilon}(0)$.  Also, $|f-f_\epsilon|<2\epsilon.$ On $D_2$ this is because $f_\epsilon =g_2.$ On $D_1,$ this follows since
$|f-g_1|,|g_1-h_{\epsilon}|<\epsilon/3$ and $|\eta|<\epsilon$.  Since
$\epsilon$ is an arbitrary positive number, there remains only to show
that $f_\epsilon^{-1}(0)\subset\{+1\}$ and since we already know that
$g_2^{-1}(0)\subset\{+1\},$ it is sufficient to show that
$f_\epsilon(z)\ne 0$ for $z \in \overline D_1$.  We break into cases
$|z| \geq \delta_2$ and $|z| \leq \delta_2$.

Let $m=\min\{|h_\epsilon(z)|:z\in\overline D_1\}.$ We now choose
$\delta_3$ so small that $\eta$ maps the region $\{z \in \overline
D_1: |z|\geq \delta_2\}$ into the set $\{w: |w|< m\}$.
%%  let $\pm p$ be the
%% two points $(|z|=\delta_2)\cap\partial D_1$ and let
%% $\zeta_\pm=\eta(\pm p).$ We now choose $\delta_3$ so small that
%% $|\zeta_\pm|<m.$
Then, for $z \in \overline D_1, |z|\geq \delta_2$,
\[
 |f_\epsilon(z)| \ge |h_\epsilon(z)|-|\eta(z)| >m-m = 0.
\]

%% To show $0\not\in f_\epsilon(\overline D_1),$ we consider first the
%% case that $z\in \overline D_1\cap\{|z|\ge\delta_2\}.$ Since $\eta$
%% maps the circular arc $\overline D_2\cap(|z|=\delta_2)$ to the
%% circular arc $|\eta(z)|=|\zeta_\pm|,$ we have

Now, suppose $z\in\overline D_1, |z|\leq\delta_2.$ Then
$$ f_\epsilon(z) = h_\epsilon(z)+\eta(z) =
h_\epsilon(z)+te^{i(\alpha+\varphi)},
$$ with $|\varphi|\le\delta_1$ and $t\le r,$ because the lens
$\eta(\overline D_1)$ lies in the pie piece $|w|\le r, |\arg
w-\alpha|\le\delta_1.$ Thus, by \eqref{zero-free-pie},
$f_\epsilon(z)\not=0.$

We have shown that $f_\epsilon^{-1}(0)\subset\{+1\}.$ Since
$\eta(-1)=0,$ we also have that $f_\epsilon(-1)=f(-1)$ and moreover
$f_\epsilon(+1)=g_2(+1)=f(+1).$ This concludes the proof.
\end{proof}

\begin{lemma}\label{disc-chain}
Let $D=D_1\cup D_2\cup\cdots\cup D_n,$ where the $D_j$ are discs of
radius $1/2$ whose respective centers are the points
$1/2,3/2,\cdots,(2n-1)/2$ and whose points of tangency are
$1,2,\cdots,n-1.$ Suppose $f\in A(\overline D)$ and $f(z)\not=0,$ for
$z\in D.$ Then, for each $\epsilon>0,$ there is an $f_\epsilon\in
A(\overline D)$ such that $|f-f_\epsilon|<\epsilon$ and
$f_\epsilon(z)\not=0,$ for $z\in\overline D.$
\end{lemma}

\begin{proof}
Set $f_j=f\mid\overline D_j.$ By Lemma \ref{shrinking}, we may assume
that $f_1^{-1}(0)\subset\{1\}$ (by ``shrinking toward 1'') and
$f_n^{-1}(0)\subset\{n-1\}$ (by ``shrinking toward $n-1$'').

By Lemma \ref{lens}, we may assume that, for $j=2,3,\cdots,n-1,$ we
have $f_j^{-1}(0)\subset\{j-1,j\}.$

Now, we proceed by finite induction to eliminate the only possible
remaining zeros $1,2,\cdots,n-1.$ Applying Lemma \ref{2discs} to
$D_1\cup D_2,$ we may get rid of the the possible zero $1.$ Then,
applying Lemma \ref{2discs} to $D_2\cup D_3,$ we get rid of the
possible zero $2.$ After $n-1$ steps, we have eliminated all possible
zeros. This concludes the proof of the lemma.
\end{proof}

\begin{proof}[Proof of Theorem]
It is sufficient to approximate $f$ uniformly by a function
$f_\epsilon\in A(\overline\Omega)$ such that $f_\epsilon(z)\not=0,$ for
$z\in\overline\Omega,$ since such an $f_\epsilon$ can in turn be
uniformly approximated by polynomials which are zero-free on
$\overline\Omega$ by Mergelyan's theorem.

For each $j=1,2,\cdots,n,$ let $\phi_j(w)=z$ be a conformal mapping of
the disc $D_j$ from the previous lemma onto the Jordan domain
$\Omega_j.$ By the Osgood-Carath\'eodory Theorem, $\phi_j$ extends to
a homeomorphism of $\overline D_j$ onto $\overline\Omega_j$ and we may
assume that $\phi_j$ maps the points of tangency of $D_j$ with
neighboring discs to the points of tangency of $\Omega_j$ with
neighboring Jordan domains. Let $\phi$ be the map from $\overline D$
to $\overline\Omega,$ defined by setting $\phi=\phi_j$ on $\overline
D_j.$ Setting $g=f\circ\phi,$ we have $g\in A(\overline D)$ and
$g(w)\not=0$ for $w\in D.$ By Lemma \ref{disc-chain}, there is a
$g_\epsilon\in A(\overline D)$ such that $|g-g_\epsilon|<\epsilon$ and
$g_\epsilon(w)\not=0,$ for $w\in\overline D.$ We may set
$f_\epsilon(z)=g(\phi^{-1}(z))=g(w).$
\end{proof}

{\bf Acknowledgement.} This paper originated in the workshop:
Stability, hyperbolicity, and zero localization of functions, December
5-9, 2011, which was held at the American Institute of Mathematics
(AIM) and which was organized by Petter Branden, George Csordas, Olga
Holtz, and Mikhail Tyaglov. Sergei Kruschev took part in the initial
discussions of the present paper. We thank all of these people as well
as the staff of AIM.

\end{document}